\documentclass[12pt]{amsart}
\usepackage{amssymb}
\usepackage{amsmath, amscd}
\usepackage{amsthm}
\usepackage{comment}
\usepackage[table,xcdraw]{xcolor}
\usepackage{ulem}
\usepackage{tikz}
\usepackage{float}
\usepackage{graphicx}
\usepackage{circuitikz}
\usetikzlibrary{positioning}
\usepackage[colorlinks=true, linkcolor=blue,urlcolor=blue]{hyperref}

\newtheorem{theorem}{Theorem}[section]

\newtheorem{proposition}[theorem]{Proposition}
\newtheorem{lemma}[theorem]{Lemma}

\newtheorem{corollary}[theorem]{Corollary}
\theoremstyle{definition}

\newtheorem{example}[theorem]{Example}
\newtheorem{definition}[theorem]{Definition}

\newtheorem{problem}[theorem]{Problem}

\newcommand{\bigzero}{\mbox{\normalfont\Large\bfseries 0}}

\begin{document}

\author[P. Danchev]{Peter Danchev}
\address{Institute of Mathematics and Informatics, Bulgarian Academy of Sciences, 1113 Sofia, Bulgaria}
\email{danchev@math.bas.bg; pvdanchev@yahoo.com}
\author[O. Hasanzadeh]{Omid Hasanzadeh}
\address{Department of Mathematics, Tarbiat Modares University, 14115-111 Tehran Jalal AleAhmad Nasr, Iran}
\email{hasanzadeomiid@gmail.com}
\author[A. Javan]{Arash Javan}
\address{Department of Mathematics, Tarbiat Modares University, 14115-111 Tehran Jalal AleAhmad Nasr, Iran}
\email{a.darajavan@gmail.com}
\author[A. Moussavi]{Ahmad Moussavi}
\address{Department of Mathematics, Tarbiat Modares University, 14115-111 Tehran Jalal AleAhmad Nasr, Iran}
\email{moussavi.a@modares.ac.ir; moussavi.a@gmail.com}

\title[Rings whose non-invertible elements are strongly nil-clean]{\small Rings whose non-invertible elements are \\ strongly nil-clean}

\maketitle

%\date{.}

%\begin{document}

%\maketitle

\begin{abstract}
We consider in-depth and characterize in certain aspects those rings whose non-units are strongly nil-clean in the sense that they are a sum of commuting nilpotent and idempotent. In addition, we examine those rings in which the non-units are uniquely nil-clean in the sense that they are a sum of a nilpotent and an unique idempotent. In fact, we succeeded to prove that these two classes of rings can completely be characterized in terms of already well-studied and fully described sorts of rings.
\end{abstract}

\section{Introduction and Basic Concepts}

Everywhere in the current paper, let $R$ be an associative but {\it not} necessarily commutative ring having identity element, usually stated as $1$. Standardly, for such a ring $R$, the letters $U(R)$, $\rm{Nil}(R)$ and ${\rm Id}(R)$ are designed for the set of invertible elements (also termed as the unit group of $R$), the set of nilpotent elements and the set of idempotent elements in $R$, respectively. Likewise, $J(R)$ denotes the Jacobson radical of $R$, and ${\rm Z}(R)$ denotes the center of $R$. The ring of $n\times n$ matrices over $R$ and the ring of $n\times n$ upper triangular matrices over $R$ are denoted by ${\rm M}_{n}(R)$ and ${\rm T}_{n}(R)$, respectively. Standardly, a ring is said to be {\it abelian} if each of its idempotents is central, that is, ${\rm Id}(R)\subseteq {\rm Z}(R)$.

In order to present our achievements here, we now need the necessary background material as follows: Mimicking \cite{4}, an element $a$ from a ring $R$ is called {\it clean} if there exists $e\in {\rm Id}(R)$ such that $a-e\in U(R)$. Then, $R$ is said to be {\it clean} if each element of $R$ is clean. In addition, $a$ is called {\it strongly clean} provided $ae=ea$ and, if each element of $R$ are strongly clean, then $R$ is said to {\it strongly clean} too. On the other hand, imitating \cite{25}, $a\in R$ is called {\it uniquely clean} if there exists a unique $e \in {\rm Id}(R)$ such that $a-e \in U(R)$. In particular, a ring $R$ is said to be {\it uniquely clean} (or just {\it UC} for short) if every element in $R$ is uniquely clean. Generalizing these notions, in \cite{19} was defined an element $a \in R$ to be {\it uniquely strongly clean} if there exists a unique $e \in {\rm Id}(R)$ such that $a-e \in U(R)$ and $ae=ea$. In particular, a ring $R$ is {\it uniquely strongly clean} (or just {\it USC} for short) if each element in $R$ is uniquely strongly clean. A ring $R$ is generalized uniquely clean (or just GUC for short) if every non-invertible element of $R$ is uniquely clean, which was introduced in \cite{2}. A ring is called a {\it generalized uniquely strongly clean} ring (or just {\it GUSC} for short) if every non-invertible element is uniquely strongly clean. These rings are generalization of USC rings, which was introduced in \cite{20}.

Let $R$ be a ring. An element $r \in R$ is said to be {\it nil-clean} if there is an idempotent $e \in R$ and a nilpotent $b \in R$ such that $r=e+b$. Such an element $r$ is further called {\it strongly nil-clean} if the existing idempotent and nilpotent can be chosen such that $be=eb$. A ring is called {\it nil-clean} (respectively, {\it strongly nil-clean}) if each of its elements is nil-clean (respectively, strongly nil-clean). Nil-clean and strongly nil-clean rings was introduced by Diesl in \cite{21}. An element $a$ in a ring $R$ is called uniquely nil-clean (or just UNC for short) if there is a unique idempotent $e$ such that $a-e$ is nilpotent. We will say that a ring is uniquely nil-clean (or just UNC for short) if each of its elements is uniquely nil-clean. These rings were also introduced by Diesl in \cite{21}. A ring $R$ is called a $UU$ ring if $U(R) = 1 + Nil(R)$, which was introduced by Calugareanu \cite{16} and studied in more details by Danchev and Lam in \cite{17}. Diesl in \cite{21} proved that a unit $u$ of $ R $ is strangly nil-clean if and only if $ u\in 1+Nil(R) $. In particular, $R$ is a $UU$ ring if and only if every unit of $ R $ is strongly nil-clean. It is clear that the $UU$ rings are generalization of strongly nil-clean rings. Also, Karimi-Mansoub et al in \cite{8} proved that a ring $R$ is a $UU$ if and only if every unit of $R$ is uniquely nil-clean. It is also clear that $UU$ rings are  generalization of uniquely nil-clean rings. So, this idea comes to mind that what can be said about rings whose non-invertible elements are strongly nil-clean and rings whose non-invertible elements are uniquely nil-clean. Also we know that $UU$ ring need not be strongly clean. Thus, a natural problem is what generalizations of strongly nil-clean and uniquely nil-clean rings can be found that are strongly clean. In this paper, we introduce two families of rings. The first one is a generalization of uniquely nil-clean rings which is a subclass of strongly clean rings and, the second one is a generalization of strongly nil-clean rings which are strongly $\pi$-regular and strongly clean. These families include rings in which each non-invertible element is uniquely nil-clean (or just GUNC for short) and rings in which every non-invertible element is strongly nil-clean (or just GSNC for short). Various extensions of these rings will be studied.

We are now planning to give a brief program of our results established in the sequel: In the next second section, we establish some fundamental characterization properties of GUNC rings -- for instance, we succeeded to establish a valuable necessary and sufficient condition, which totally classifies any ring to be GUNC (see Theorem~\ref{.}). In the subsequent third section, we explore GUNC group rings and obtain a good criterion for a group ring of a locally finite $p$-group, with $p$ a prime, over an arbitrary ring to be GUNC. In the next fourth section, we give a comprehensive investigation of GSNC rings and characterize them in several ways (see, e.g., Theorems~\ref{thm3.9}, \ref{thm3.12}, \ref{thm3.33} and \ref{thm3.35}, respectively). Our fifth section is devoted to the examination in-depth of GSNC group rings and we receive some satisfactory characterization of their structure. We finish our study in the sixth section with two intriguing left-open questions that are of some interest and importance.

\section{GUNC rings}

We start here with the following key notion.

\begin{definition}\label{def2.1}
A ring $R$ is called {\it generalized uniquely nil-clean} (or just {\it GUNC} for short) if every non-invertible element in $R$ is uniquely nil-clean.	
\end{definition}

We now have the following diagram:

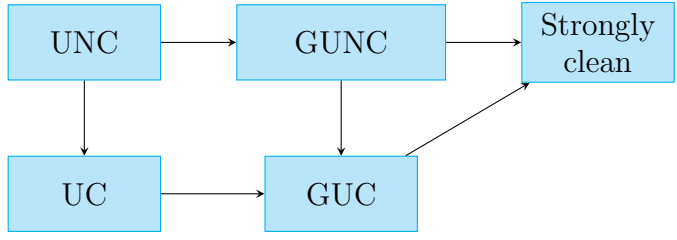
\begin{center}
\begin{tikzpicture}
\node[draw=cyan,fill=cyan!25,minimum width=2cm,minimum height=1cm,text width=1.75cm,align=center]  (a) {UNC};	 \node[draw=cyan,fill=cyan!25,minimum width=2cm,minimum height=1cm,text width=2.5cm,align=center,right=of a](b){GUNC};	 \node[draw=cyan,fill=cyan!25,minimum width=2cm,minimum height=1cm,text width=1.75cm,align=center,below=of b](c){GUC};	 \node[draw=cyan,fill=cyan!25,minimum width=2cm,minimum height=1cm,text width=1.75cm,align=center,right=of b](d){Strongly clean};
\node[draw=cyan,fill=cyan!25,minimum width=2cm,minimum height=1cm,text width=1.75cm,align=center,below=of a](f){UC};
\draw[-stealth] (a) -- (b);
\draw[-stealth] (b) -- (c);
\draw[-stealth] (a) -- (f);
\draw[-stealth] (f) -- (c);
\draw[-stealth] (b) -- (d);
\draw[-stealth] (c) -- (d);
\end{tikzpicture}
\end{center}

The next example gives us the opportunity to discover the complicate structure of these rings.

\begin{example}\label{exa2.2}
(i) Any UNC ring is GUNC, but the converse is {\it not} true in general. In fact, a simple check shows that the ring $\mathbb{Z}_{3}$ is GUNC that is {\it not} UNC.

\medskip

(ii) Any UNC ring is UC, but the converse is {\it not} true in general. In fact, a plain verification shows that the ring $\mathbb{Z}_{4}[[x]]$ is UC that is {\it not} UNC.

\medskip

(iii) Any UC ring is GUC, but the converse is {\it not} true in general. Indeed, an easy inspection shows that the ring $\mathbb{Z}_{5}$ is GUC that is {\it not} UC.

\medskip

(iv) Any GUC ring is strongly clean, but the converse is {\it not} true in general. Indeed, a quick trick shows that the ring ${\rm M}_{2}(R)$, where $R=\left\lbrace  \left( \begin{array}{cc}
	a & b \\
	0 & a
\end{array}\right)\Big| a\in\mathbb{Z}_{2},\, b\in\mathbb{Z}_{(2)}[x]
\right\rbrace$, is strongly clean that is {\it not} GUC.

This substantiates our argumentation.
\end{example}	

WE now need a series of preliminary technicalities.

\begin{lemma}\label{lem2.3}
Let $R$ be a ring and $a\in R$. Then, $a$ is UNC if, and only if, $1-a$ is UNC.
\end{lemma}

\begin{proof}
Given $a\in R$ is UNC. Then, there exists $e^{2}=e\in R$ and  $q\in {\rm Nil}(R)$ such that $a=e+q$. Hence, $1-a=(1-e)+(-q)$. Suppose $1-a=f+q^{\prime}$, where $f^{2}=f\in R$ and $q^{\prime}\in {\rm Nil}(R)$. Thus,
$a=(1-f)+(-q^{\prime})$. Since $a$ is UNC, one sees that $1-f=e$ whence $1-e=f$, so $1-a$ is too UNC. The converse is similar, so we omit the details.
\end{proof}

\begin{lemma}\label{lem2.4}
Let $R$ be a GUNC ring. Then, $R$ is abelian.
\end{lemma}

\begin{proof}
Given $e\in Id(R)$. If $e\in U(R)$, it must be that $e=1$ and hence $e$ is central. If, however, $e\notin U(R)$,
it follows that $e$ is UNC and, therefore, we find that $e$ is central by \cite[Corollary 2.4]{7}, as required.
\end{proof}

\begin{lemma}\label{lem2.5}
If the direct product $\prod_{i=1}^{n}R_{i}$ is a GUNC ring, then each direct component $R_i$ is a GUNC ring.
\end{lemma}

\begin{proof}
Given $a\in R_{i}$, where $a\notin U(R_{i})$, so one sees that the vector $(1,\ldots,a,1,\ldots,1)$ is not a unit in
$\prod_{i=1}^{n}R_{i}$. However, as $\prod_{i=1}^{n}R_{i}$ is GUNC, the element $(1,\ldots,a,1,\ldots,1)$ is UNC,
so that $a$ is UNC as well; for otherwise, if $a$ has two different nil-clean decompositions, then $(1,\ldots,a,1,\ldots,1)$ will also have two different nil-clean decompositions, which is a contradiction, as pursued.
\end{proof}

It is worthy of noticing that the converse claim of Lemma \ref{lem2.5} is {\it not} true. For instance, the ring $\mathbb{Z}_{3}$ is obviously GUNC, but however the direct product $\mathbb{Z}_{3}\times\mathbb{Z}_{3}$ is {\it not} GUNC.

Nevertheless, we can offer the following.

\begin{proposition}\label{pro2.6}
Let $R_{i}$ be rings for all $1\leqslant i \leqslant n$. Then, the product $\prod_{i=1}^{n}R_{i}$ is a GUNC ring if, and only if, each direct factor $R_{i}$ is a UNC ring.
\end{proposition}

\begin{proof}
Assuming every $R_{i}$ is UNC, so $\prod_{i=1}^{n}R_{i}$ is UNC using \cite[Proposition 5.2]{21} and hence $\prod_{i=1}^{n}R_{i}$ is necessarily GUNC.

Conversely, assume that $\prod_{i=1}^{n}R_{i}$ is GUNC and, in a way of contradiction, that $R_{j}$ is {\it not} UNC for some index $j$. Then, there exists $a\in R_{j}$ which is not UNC, so the vector $(0,\ldots,0,a,0,\ldots,0)$
is not UNC in $\prod_{i=1}^{n}R_{i}$ too. But, one verifies that $$(0,\ldots,0,a,0,\ldots,0)\notin U(\prod_{i=1}^{n}R_{i})$$ and, by hypothesis, $(0,\ldots,0,a,0,\ldots,0)$ is UNC contradicting our assumption. Therefore, every $R_{i}$ is UNC, as claimed.
\end{proof}

A ring $R$ is called {\it division}, if every non-zero element of $R$ is invertible. Also, a ring $R$ is called {\it local} if $R/J(R)$ is a division ring. thereby, as a consequence, we yield:

\begin{corollary}\label{cor2.7}
Let $R$ be a ring, and $e^{2}=e\in {\rm Z}(R)$. If $R$ is GUNC, then $eRe$ is GUNC. In particular, if $e$ is non-trivial, then $eRe$ is UNC.
\end{corollary}

\begin{proof}
If, for a moment, a GUNC ring $R$ is {\it not} a local ring, then there exists an idempotent $e$ which is {\it not} trivial such that $R=eRe\oplus(1-e)R(1-e)$. As $R$ is GUNC, the corner $eRe$ is GUNC in accordance with Lemma
\ref{lem2.5}. In addition, if $e$ is non-trivial, then the subring $eRe$ is UNC in view of Proposition \ref{pro2.6}.
\end{proof}

Recall that a ring $R$ is {\it directly finite}, provided that $ab=1$ implies $ba=1$ for all $a,b\in R$ (or, equivalently, $aR=R$ implies $Ra=R$). We, thus, arrive at the following interesting property of GUNC rings.

\begin{proposition}\label{pro2.8}
Every GUNC ring is directly finite.
\end{proposition}

\begin{proof}
Letting $ab=1$, it must be that $(ba)^{2}=baba=ba$. So, $ba$ is an idempotent in $R$, and hence it is central by Lemma \ref{lem2.4}. Therefore, $$ba=ba(ab)=a(ba)b=(ab)(ab)=1,$$ as required.
\end{proof}

Our next machinery, necessary to establish the global results, is the following.

\begin{proposition}\label{pro2.9}
Let $R$ be a ring. If $R$ is local with $J(R)$ nil, then $R$ is GUNC.
\end{proposition}

\begin{proof}
Supposing $a\in R$, where $a\notin U(R)$, so $a\in J(R)\subseteq {\rm Nil}(R)$. Then, $a$ has the only nil-clean expression like this $a=0+a$. Hence, $a$ is UNC, as needed.
\end{proof}

\begin{lemma}\label{lem2.10}
Let $R$ be a ring. Then, the following are equivalent:
\\
$(i)$ $R$ is either local with $J(R)$ nil, or $R$ is UNC.
\\
$(ii)$ $R$ is a GUNC ring.
\end{lemma}

\begin{proof}
(i) $\Rightarrow$ (ii). It is straightforward bearing in mind Proposition \ref{pro2.9}.
\\
(ii) $\Rightarrow$ (i). Assuming that $R$ is a GUNC ring which is {\it not} local, then there exists an idempotent $e$ that is {\it not} trivial such that $R=eRe\oplus(1-e)R(1-e)$. Hence, both $eRe$ and $(1-e)R(1-e)$
are UNC taking into account Corollary \ref{cor2.7}. Consequently, $R$ is UNC in virtue of \cite[proposition 5.2]{21}.
\end{proof}

\begin{corollary}\label{cor2.11}
A ring $R$ is GUNC if, and only if, $R$ is GUC and $J(R)$ is nil.
\end{corollary}

\begin{proof}
It follows combining Lemma \ref{lem2.10}, \cite[Theorem $ 2.10 $ ]{2} and \cite[Lemma 5.3.7]{3}.
\end{proof}

\begin{corollary}\label{cor2.12}
If $R$ is GUNC, then $J(R)$ is nil.
\end{corollary}

\begin{proof}
It is immediate by combination of Lemma \ref{lem2.10} and \cite[Theorem $ 5.9 $]{21}.
\end{proof}

A ring $R$ is called {\it boolean} if every element of $R$ is an idempotent.

\begin{corollary}\label{cor2.13}
If $R$ is GUNC, then $J(R)={\rm Nil}(R)$.
\end{corollary}

\begin{proof}
Consulting with Corollary \ref{cor2.12}, we know that $J(R)\subseteq {\rm Nil}(R)$. Now, assume that $x\in {\rm Nil}(R)$. Then, $\bar{x}\in \bar{R}=\dfrac{R}{J(R)}$ is nilpotent. Exploiting \cite[Corollary $ 2. 11 $]{2},
the quotient-ring $\bar{R}$ is either boolean or division. If, foremost, $\bar{R}$ is boolean, thus $\bar{x}$ has to be an idempotent, whence $\bar{x}=\bar{0}$. So, $x\in J(R)$.

If, however, $\bar{R}$ is a division factor-ring, we have again $\bar{x}=\bar{0}$ and $x\in J(R)$. Thus, in both cases, $x\in J(R)$ and hence $J(R)={\rm Nil}(R)$, as stated.
\end{proof}

The following property sounds somewhat curiously.

\begin{corollary}\label{cor2.14}
If $R$ is GUNC, then $R$ is strongly clean.
\end{corollary}

\begin{proof}
It follows at once applying Corollary \ref{cor2.11} and \cite[Lemma $2.3$]{2}.
\end{proof}

The next statement is pivotal for our further presentation.

\begin{proposition}\label{pro2.15}
Let $R$ be a ring and $I$ is a nil-ideal of $R$. Then, the following two equivalencies hold:
\\
$(i)$ $R$ is GUNC if, and only if, $\dfrac{R}{J(R)}$ is GUNC, $J(R)$ is nil, and $R$ is abelian.
\\
$(ii)$ $R$ is GUNC if, and only if, $\dfrac{R}{I}$ is GUNC and $R$ is abelian.
\end{proposition}

\begin{proof}
(i). Given $R$ is GUNC and $\bar{a}\in \bar{R}=\dfrac{R}{J(R)}$, where $\bar{a}\notin U(R)$. So, one sees that $a\notin U(R)$, as for otherwise, if $a\in U(R)$, then it must be that $\bar{a}\in U(\bar{R})$ and this is a contradiction. Henceforth, we must show that $\bar{a}$ is UNC. To this goal, write $\bar{a}=\bar{e}+\bar{q}_{1}=\bar{f}+\bar{q}_{2}$, where $\bar{e},\bar{f}\in Id(\bar{R})$ and $\bar{q}_{1},\bar{q}_{2}\in Nil(\bar{R})$.
Thus, we have $a-(e+q_{1}), a-(f+q_{2})\in J(R)$ and hence $a=e+(q_{1}+j_{1})=f+(q_{2}+j_{2})$ for some $j_{1},j_{2}\in J(R)$, where $e,f\in Id(R)$, because idempotents lift modulo $J(R)$, and $(q_{1}+j_{1}),(q_{2}+j_{2})\in {\rm Nil}(R)$. But we know that $a$ is UNC, and so $e=f$ whence $\bar{e}=\bar{f}$. Therefore, $\bar{a}$ is UNC. Moreover, $J(R)$ is nil owing to Corollary \ref{cor2.12}, and $R$ is abelian according to Lemma \ref{lem2.4}, as formulated.

Conversely, let $a\in R$, where $a\notin U(R)$. So, one observes that $\bar{a}\notin U(\bar{R})$, as for otherwise, if $\bar{a}\in U(\bar{R})$, then it must be that $a\in U(R)$, because units lift modulo $J(R)$ and this is a contradiction. Now, writing $a=e+q_{1}=f+q_{2}$, where $e,f\in {\rm Id}(R)$ and $q_{1},q_{2}\in {\rm Nil}(R)$, so we have $\bar{a}=\bar{e}+\bar{q}_{1}=\bar{f}+\bar{q}_{2}$, where $\bar{e},\bar{f}\in {\rm Id}(\bar{R})$ and
$\bar{q}_{1},\bar{q}_{2}\in {\rm Nil}(\bar{R})$. But, the element $\bar{a}$ is UNC, so that $\bar{e}=\bar{f}$ and, consequently, $e-f\in J(R)$. As $R$ is abelian, $(e-f)^{3}=(e-f)$, and thus $e-f=0$ implying $e=f$. Finally, $a$ is UNC, as expected.
\\
(ii). The proof is quite similar to the preceding point (i), so we omit the details.
\end{proof}

\begin{proposition}\label{pro2.16}
$ ~ $
\\
$(i)$ For any ring $R$, the power series ring $R[[x]]$ is not GUNC.
\\
$(ii)$ If $R$ is any commutative ring, then $R[x]$ is not GUNC.
\\
$(iii)$ The matrix rings ${\rm M}_{n}(R)$ and ${\rm T}_{n}(R)$ are never GUNC for any $n\geqslant 2$.
\end{proposition}

\begin{proof}
(i). Note the principal fact that the Jacobson radical of $R[[x]]$ is {\it not} nil. Thus, invoking Corollary \ref{cor2.12}, $R[[x]]$ is really {\it not} a GUNC ring.
\\
(ii). If we assume the contrary that $R[x]$ is GUNC, then Corollary \ref{cor2.14} gives that $R[x]$ is clean. This, however, is impossible in conjunction with \cite[Example $ 2 $]{4}.
\\
(iii). It is pretty obvious referring to Lemma \ref{lem2.4}.
\end{proof}

The next constructions are worthwhile.

\begin{example}\label{exa2.17}
$ ~ $
\\
(i). Any field and even any division ring is GUNC.
\\
(ii). Any GUNC ring is strongly clean but, the converse is manifestly {\it not} true in general. Indeed, the ring ${\rm M}_{2}(\hat{\mathbb{Z}}_{p})$ is strongly clean, but is definitely {\it not} GUNC.
\\
(iii). Any GUNC ring is GUC but, the converse is {\it not} true in general. In fact, it is {\it not} too hard to see that the ring $\mathbb{Z}_{2}[[x]]$ is GUC but is {\it not} GUNC.
\end{example}

\begin{proof}
(i). It is evident by the definition of a GUNC ring.
\\
(ii). Note that, adapting \cite[Theorem $ 2. 4 $]{18}, the ring ${\rm M}_{2}(\hat{\mathbb{Z}}_{p})$ is strongly clean. However, it is {\it not GUNC}, because all GUNC rings are always abelian.
\\
(iii). Note that the ring $\mathbb{Z}_{2}[[x]]$ is uniquely clean, and hence is GUC, but it is {\it not} GUNC as Proposition \ref{pro2.16} tell us.
\end{proof}

We now come to the following necessary and sufficient condition.

\begin{proposition}\label{pro2.18}
A ring $R$ is UNC if, and only if, $R$ is simultaneously GUNC and UU.
\end{proposition}

\begin{proof}
Given $R$ is UNC and $u\in U(R)$, so one may write that $u=e+q$, where $e\in {\rm Id}(R)$ and $q\in {\rm Nil}(R)$. Thus, $e=u-q$. But \cite[Theorem $ 5. 3. 3 $]{3} informs us that every UNC ring is abelian.
Therefore, $e\in U(R)$ and hence $e=1$. Then, $u\in 1+{\rm Nil}(R)$ and $R$ is a UU ring.

For the converse implication, it suffices to show that $R$ is UU if, and only if, every unit of $R$ is UNC. However, this has been proven in \cite[Theorem $ 2. 23 $]{8}.
\end{proof}

Recall that a ring is called {\it reduced} if it has no non-zero nilpotent elements.

\begin{lemma}\label{lem2.19}
Let $R$ be a GUNC ring and $J(R)=\{0\}$. Then, $R$ is reduced.
\end{lemma}

\begin{proof}
Assume that $x^{2}=0$, where $0\neq x\in R$. Since $R$ is GUNC, we know that $R$ is clean, and hence $R$ is semi-potent thanks to \cite[Proposition $ 2. 16 $]{2}. Thus, consulting with \cite[Theorem $ 2.1 $]{6}, there exists
$0\neq e^{2}=e\in R$ such that $eRe\cong {\rm M}_{2}(S)$ for some non-trivial ring $S$. But, in regard to Proposition \ref{pro2.16}, ${\rm M}_{2}(S)$ is {\it not} a GUNC ring. That is why, $eRe$ is {\it not} GUNC, and this contradicts Corollary \ref{cor2.7}. Finally, $R$ is reduced, as promised.
\end{proof}

Our final assertion for this section, which states an interesting criterion for a ring to be GUNC, is the following one.

\begin{proposition}\label{pro2.20}
A ring $R$ is GUNC if, and only if, all next three conditions are fulfilled:
\\
$(i)$ ${\rm Nil}(R)$ is an ideal of $R$;
\\
$(ii)$ $\frac{R}{{\rm Nil}(R)}$ is either a boolean ring, or a division ring;
\\
$(iii)$ $R$ is an abelian ring.
\end{proposition}

\begin{proof}
"$\Rightarrow$". With Corollary \ref{cor2.13}, Corollary \ref{cor2.11}, \cite[Corollary $ 2.11 $]{2} and Lemma \ref{lem2.4} at hand, all things are rather easy.
\\
"$\Leftarrow$". Firstly, letting $\frac{R}{{\rm Nil}(R)}$ is a boolean ring, we show that $R$ is UNC. To this purpose, chosen $ a\in R $, so $a-a^{2}\in {\rm Nil}(R)$. By hypothesis, there exists a unique idempotent, $e\in R$ say,
such that $q:=a-e\in {\rm Nil}(R)$. Hence, $a=e+q$. Also, write $a=f+q$, where $f\in {\rm Id}(R)$ and $q^{\prime}\in {\rm Nil}(R)$. However, treating the uniqueness, we get $e=f$. Therefore, $R$ is UNC, and so GUNC.

Now, let $\frac{R}{{\rm Nil}(R)}$ is a division ring, it is readily to see that $\frac{R}{{\rm Nil}(R)}$ is GUNC. Then, $R$ is GUNC using Proposition \ref{pro2.15}, as wanted.
\end{proof}

\section{GUNC group rings}

We know with the help of Lemma \ref{lem2.10} that GUNC rings include both UNC rings and local rings, and vice versa. In addition, there are some results on UC group rings and local group rings proved in \cite{29} and \cite{27}, respectively. Correspondingly, we can also obtain some achievements about GUNC group rings that could be of some interest and importance. To this aim, we recall that a group $G$ is a {\it $p$-group} if every element of $G$ is a power of $p$, where $p$ is a prime. Likewise, a group $G$ is called {\it locally finite} if every finitely generated subgroup is finite.

Suppose now that $G$ is an arbitrary group and $R$ is an arbitrary ring. As usual, $RG$ stands for the group ring of $G$ over $R$. The homomorphism $\varepsilon :RG\rightarrow R$, defined by $\varepsilon (\displaystyle\sum_{g\in G}a_{g}g)=\displaystyle\sum_{g\in G}a_{g}$, is called the {\it augmentation map} of $RG$ and its kernel, denoted by $\Delta (RG)$, is called the {\it augmentation ideal} of $RG$.

Our two affirmations, motivated us in writing this section, are these:

\begin{proposition}\label{pro3.1}
Let $R$ be a ring and let $G$ be a group. If $RG$ is a GUNC ring, then $R$ is GUNC, $G$ is a $p$-group and $p\in {\rm Nil}(R)$.
\end{proposition}

\begin{proof}
Assume $RG$ is GUNC. Then, exploiting Lemma \ref{lem2.10}, either $RG$ is local with $J(RG)$ nil, or $RG$ is an UNC ring. We consider these two possibilities in the sequel:
	
(1) If $RG$ is a local ring with $J(RG)$ nil, then \cite[Corollary]{27} guarantees that $\Delta(RG) \subseteq J(RG)$. Hence, \(\Delta(RG)\) is a nil-ideal. Therefore, with \cite[Proposition 16]{28} at hand, one deduces that $G$ is a $p$-group, where $p \in {\rm Nil}(R)$. Moreover, since $RG/\Delta(RG) \cong R$, the application of Proposition \ref{pro2.15} ensures that $R$ is a GUNC ring, as desired.
	
(2) If $RG$ is a UNC ring, then $RG$ is obviously uniquely clean. Applying \cite[Theorem 5]{29}, one infers that $G$ is a $2$-group, and employing \cite[Corollary 5.3.5]{3}, one finds that $R$ is an UNC ring. Furthermore, \cite[Proposition 3.14]{21} assures that $2 \in {\rm Nil}(R)$, as asked for.
\end{proof}

\begin{proposition}\label{pro3.2}
Let $R$ be a ring and let $G$ be a locally finite $p$-group. Then, $RG$ is GUNC if, and only if, $R$ is GUNC and $p\in {\rm Nil}(R)$.
\end{proposition}

\begin{proof}
Assume $R$ is a GUNC ring and $p \in {\rm Nil}(R)$. Then, utilizing \cite[Proposition 16]{28}, one derives that $\Delta(RG)$ is a nil-ideal. Furthermore, Lemma \ref{lem2.10} insures that either $R$ is local with nil Jacobson radical, or $R$ is an UNC ring.
	
If, firstly, $R$ is local with nil Jacobson radical, then \cite[Corollary]{27} is a guarantor that $RG$ is local. Thus, $\varepsilon(J(RG)) \subseteq J(R)$, because $\varepsilon$ is an onto ring homomorphism. Since $J(R)$ is nil, for every $f \in J(RG)$, there exists $k \in \mathbb{N}$ such that $f^k \in \Delta(RG)$. Likewise, since $\Delta(RG)$ is nil, we have $f$ is nil. Hence, $RG$ is a local ring with nil Jacobson radical, and forcing Lemma \ref{lem2.10}, we conclude that $RG$ is a GUNC ring.
	
If now $R$ is an UNC ring, then \cite[Proposition 3.14]{21} yields that $2 \in {\rm Nil}(R)$, and so $p = 2$. Therefore, \cite[Theorem 12]{29} reflects to get that $RG$ is an UNC ring, as pursued.
\end{proof}

\section{GSNC rings}

We begin here with the following key concept.

\begin{definition}\label{def3.1}
A ring $R$ is called {\it generalized strongly nil-clean} (or just {\it GSNC} for short) if every non-invertible element in $R$ is strongly nil-clean.	
\end{definition}

We now have the following diagram:

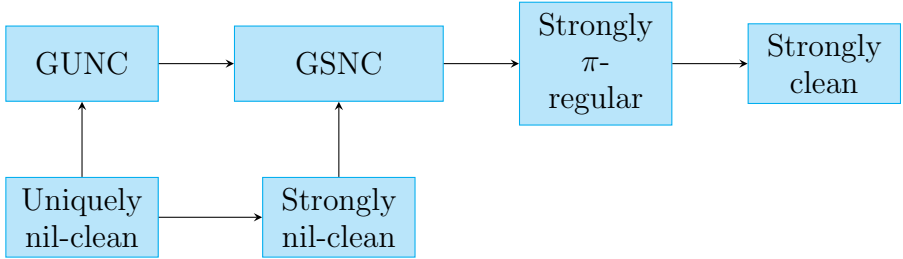
\begin{center}
\begin{tikzpicture}
\node[draw=cyan,fill=cyan!25,minimum width=2cm,minimum height=1cm,text width=1.75cm,align=center]  (a) {GUNC};
\node[draw=cyan,fill=cyan!25,minimum width=2cm,minimum height=1cm,text width=2.5cm,align=center,right=of a](b){GSNC};
\node[draw=cyan,fill=cyan!25,minimum width=2cm,minimum height=1cm,text width=1.75cm,align=center,below=of b](c){Strongly nil-clean};
\node[draw=cyan,fill=cyan!25,minimum width=2cm,minimum height=1cm,text width=1.75cm,align=center,right=of b](g){Strongly $\pi$- regular};
\node[draw=cyan,fill=cyan!25,minimum width=2cm,minimum height=1cm,text width=1.75cm,align=center,right=of g](d){Strongly clean};
\node[draw=cyan,fill=cyan!25,minimum width=2cm,minimum height=1cm,text width=1.75cm,align=center,below=of a](f){Uniquely nil-clean};
\draw[-stealth] (a) -- (b);
\draw[-stealth] (c) -- (b);
\draw[-stealth] (f) -- (a);
\draw[-stealth] (f) -- (c);
\draw[-stealth] (b) -- (g);
\draw[-stealth] (g) -- (d);
\end{tikzpicture}
\end{center}

We continue here with a series of technicalities as follows.

\begin{lemma}\label{lem3.2}
Every GSNC ring is strongly clean.
\end{lemma}

\begin{proof}
Let $a\in R$. Then, either $a\in U(R)$ or $a\notin U(R)$. If, firstly, $a\in U(R)$, then $a$ is a strongly clean element. Now, if $a\notin U(R)$, so $a$ is strongly nil-clean element, and hence $a$ is strongly clean by \cite[Corollary $ 3.6 $]{21}.
\end{proof}

\begin{lemma}\label{lem3.3}
Let $R_{i}$ be a ring for all $i\in I$. If $\prod^{n}_{i=1}R_{i}$ is GSNC, then each $R_{i}$ is GSNC.
\end{lemma}

\begin{proof}
Let $a_{i}\in R_{i}$, where $a_{i}\notin U(R_{i})$, whence $ (1,1,\ldots,a_{i},1,\ldots,1)\notin  U(\prod^{n}_{i=1}R_{i}) $. So, $(1,1,\ldots,a_{i},1,\ldots,1)$ is strongly nil-clean, and hence $a_{i}$ is strongly nil-clean. If, however, $a_{i}$ is not strongly nil clean, we clearly conclude that $(1,1,\ldots,a_{i},1,\ldots,1)$
is not strongly nil-clean and this is a contradiction.
\end{proof}

We note that the converse of Lemma \ref{lem3.3} is manifestly false. For example, $\mathbb{Z}_{3}$ is GSNC. But the direct product $\mathbb{Z}_{3}\times\mathbb{Z}_{3}$ is {\it not} GSNC, because the element $(2,0)$ is {\it not} invertible in $\mathbb{Z}_{3}\times\mathbb{Z}_{3}$ and $(0,2)$ is really not strongly nil- clean element.

\begin{lemma}\label{lem3.4}
If $R$ is GSNC, then $J(R)$ is nil.
\end{lemma}

\begin{proof}
If we have $a\in J(R)$, then $a\notin U(R)$, so $a = e + q$, where $e = e^2 \in R$, $q \in {\rm Nil}(R)$ and $eq = qe$. Therefore, $$1 - e = (1 + q) - a \in U(R) + J(R) \subseteq U(R),$$ implying $e = 0$. This, in turn, implies that $a = q \in {\rm Nil}(R)$.
\end{proof}

\begin{proposition}\label{pro3.5}
Let $R$ be a ring, and let $a\in R$. Then, the following are equivalent:
\\
(i) $R$ is a GSNC ring.
\\
(ii) For any $a\in R$, where $a\notin U(R)$, $a-a^{2}\in {\rm Nil}(R)$.
\\
(iii) For any $a\in R$, where $a\notin U(R)$, there exists $x\in R$ such that $a-ax\in {\rm Nil}(R)$, $ax=xa$
and $x=x^{2}a$.
\end{proposition}

\begin{proof}
The proof is similar to \cite[Theorem $ 5. 1. 1 $]{3}.
\end{proof}

\begin{corollary}\label{cor3.6}
Every GSNC ring is strongly $\pi$-regular.
\end{corollary}

\begin{proof}
Let $R$ be a GSNC ring. Choose $a \in R$. If $a \in U(R)$, then $a$ is a strongly $\pi$-regular element. If, however, $a \notin U(R)$, then, by Proposition \ref{pro3.5}, there exists $k \in \mathbb{N}$ such that $(a-a^2)^k = 0$, so we have $a^k = a^{k+1}r$ for some $r \in R$. Thus, $R$ is a strongly $\pi$-regular ring, as claimed.
\end{proof}

\begin{proposition}\label{pro3.7}
(i) For any nil-ideal $I \subseteq R$, $R$ is GSNC if, and only if, $R/I$ is GSNC.

(ii) A ring $R$ is GSNC if, and only if, $J(R)$ is nil and $R/J(R)$ is GSNC.

(iii) The direct product $\prod^{n}_{i=1}R_{i}$ is GSNC if, and only if, each $ R_{i} $ is strongly nil-clean.
\end{proposition}

\begin{proof}
(i) Assume $R$ is a GSNC ring and $\overline{R} := R/I$, where $\bar{a} \notin U(\overline{R})$. Then, $a \notin U(R)$, which insures, in view of Proposition \ref{pro3.5}, that $a - a^2 \in {\rm Nil}(R)$, so $\bar{a} - \bar{a}^2 \in {\rm Nil}(\overline{R})$.

Conversely, suppose $\overline{R}$ is a GSNC ring. If $a \notin U(R)$, then $\bar{a} \notin U(\overline{R})$, and Proposition \ref{pro3.5} yields that $\bar{a} - \bar{a}^2 \in {\rm Nil}(\overline{R})$. Therefore, there exists $k \in \mathbb{N}$ such that $(a - a^2)^k \subseteq I \subseteq {\rm Nil}(R)$.

(ii) Using Lemma \ref{lem3.4} and part (i) of the proof, the proof is clear.

(iii) Letting each $ R_{i} $ be strongly nil-clean, then $ \prod^{n}_{i=1}R_{i} $ is strongly nil-clean employing \cite[Proposition $ 3. 13 $]{21}. Hence, $ \prod^{n}_{i=1}R_{i}$ is GSNC.

Conversely, assume that $\prod^{n}_{i=1}R_{i}$ is GSNC and that $R_{j}$ is not strongly nil-clean for some index
$1\leqslant j\leqslant n$. Then, there exists $a\in R_{j}$ which is not strongly nil-clean. Consequently,
$(0,\ldots,0,a,0,\ldots,0)$ is not strongly nil-clean in $\prod^{n}_{i=1}R_{i}$. But, it is readily seen that
$(0,\ldots,0,a,0,\ldots,0)$ is not invertible in $\prod^{n}_{i=1}R_{i}$ and thus, by hypothesis,
$(0,\ldots,0,a,0,\ldots,0)$  is not strongly nil-clean, a contradiction. Therefore, each $ R_{i} $ is strongly nil-clean, as needed.
\end{proof}

As a consequence, we derive:

\begin{corollary}\label{cor3.8}
Let $R$ be a ring and $0\neq e\in {\rm Id}(R)$. If $R$ is GSNC, then so is $eRe$.
\end{corollary}

\begin{proof}
Assuming $a \in eRe \setminus U(eRe)$, we have $a = ea = ae = eae$. If $a \in U(R)$, then there exists $b \in R$ such that $ab = ba = 1$, which implies $a(ebe) = (ebe)a = e$, leading to a contradiction. Therefore, $a \notin U(R)$. Hence, by Proposition \ref{pro3.5}, we have $a - a^2 \in {\rm Nil}(R) \cap eRe \subseteq {\rm Nil}(eRe)$, as required.
\end{proof}

We now possess all the machinery necessary to establish the following assertion.

\begin{theorem}\label{thm3.9}
For any ring $S \neq 0$ and any integer $n \ge 3$, the ring ${\rm M}_n(S)$ is not GSNC.
\end{theorem}

\begin{proof}
Since it is well known that ${\rm M}_3(S)$ is isomorphic to a corner ring of ${\rm M}_n(S)$ (for $n \ge 3$), it suffices to show that ${\rm M}_3(S)$ is {\it not} a GSNC ring by virtue of Corollary \ref{cor3.8}. To this target, consider the matrix
$$A =\begin{pmatrix}
        1 & 1 & 0 \\
        1 & 0 & 0 \\
        0 & 0 & 0
\end{pmatrix} \notin U({\rm M}_3(S)).$$ But a plain inspection gives us that $$A - A^2 \notin {\rm Nil}({\rm M}_3(S)),$$ as expected. Therefore, Proposition \ref{pro3.5} guarantees that $R$ cannot be a GSNC ring, as asserted.
\end{proof}

The following is worthwhile to be recorded.

\begin{example}\label{exm3.10}
The ring ${\rm M}_2(\mathbb{Z}_2)$ is a GSNC ring but is {\it not} strongly nil-clean. In fact, according to \cite[Theorem 2.7]{5}, for any arbitrary ring $R$, the ring ${\rm M}_2(R)$ cannot be strongly nil-clean. However, from the obvious equality $${\rm M}_2(\mathbb{Z}_2) = U({\rm M}_2(\mathbb{Z}_2)) \cup {rm Id}({\rm M}_2(\mathbb{Z}_2)) \cup {\rm Nil}({\rm M}_2(\mathbb{Z}_2)),$$, it is evident that ${\it M}_2(\mathbb{Z}_2)$ is a GSNC ring in conjunction with Proposition \ref{pro3.5}.
\end{example}

\begin{corollary}\label{cor3.11}
The ring ${\rm M}_2(\mathbb{Z}_{2^k})$ is a GSNC ring for each $k \in \mathbb{N}$
\end{corollary}

Our next chief statement is:

\begin{theorem}\label{thm3.12}
Let R be a $2$-primal, local and strongly nil-clean ring. Then, ${\rm M}_2(R)$ is a GSNC ring.
\end{theorem}

\begin{proof}
Since $R$ is simultaneously local and strongly nil-clean, we can write that $R/J(R) \cong \mathbb{Z}_2$, so Example \ref{exm3.10} informs us that ${\rm M}_2(R/J(R))$ is a GSNC ring. On the other hand, since $R$ is both $2$-primal and strongly nil-clean, we may write $J(R) = {\rm Nil}(R) = {\rm Nil}_{\ast}(R)$, so that from this we extract that
$${\rm M}_2(R/J(R)) = {\rm M}_2(R)/{\rm M}_2(J(R)) = {\rm M}_2(R)/{\rm M}_2({\rm Nil}_{\ast}(R)) = {\rm M}_2(R)/{\rm Nil}_{\ast}({\rm M}_2(R)),$$
and, moreover, as ${\rm Nil}_{\ast}({\rm M}_2(R))$ is a nil-ideal, Proposition \ref{pro3.7}(i) forces that ${\rm M}_2(R)$ is a GSNC ring, as inferred.
\end{proof}

In the above theorem, the condition of being local is {\it not} redundant. In order to demonstrate that, supposing $R = \mathbb{Z}_2 \times \mathbb{Z}_2$, then $R$ is $2$-primal and strongly nil-clean, but definitely ${\rm M}_2(R)$ is {\it not} a GSNC ring as a plain check shows.

Besides, it is impossible to interchange the condition of being strongly nil-clean with GSNC in the above theorem. Indeed, to illustrate that such a replacement is really {\it not} possible, assuming $R = \mathbb{Z}_3$, then $R$ is a $2$-primal, GSNC, and local ring, but an easy verification shows that ${\rm M}_2(R)$ is {\it not} a GSNC ring.

\medskip

Our next series of technical claims is like this.

\begin{lemma}\label{lem3.13}
Let ${\rm M}_2(R)$ be a GSNC ring. Then, $R$ is a strongly nil-clean ring.
\end{lemma}

\begin{proof}
Let us assume that $a \in R$. Then, one sees that $$A=\begin{pmatrix}
a & 0 \\ 0 & 0
\end{pmatrix} \notin U({\rm M}_2(R)).$$ Thus, referring to Proposition \ref{pro3.5}, we have $A-A^2 \in {\rm Nil}({\rm M}_2(R))$ and, as a result, $a - a^2 \in {\rm Nil}(R)$. Therefore, consulting with \cite[Theorem 5.1.1]{3}, we can conclude that $R$ is a strongly nil-clean ring.
\end{proof}

\begin{example}\label{exa3.14}
If $R$ is a local ring with nil $J(R)$, then $R$ is GSNC.
\end{example}

\begin{proof}
Given $a\in R$ and $a\notin U(R)$. Since $R$ is local, $a\in J(R)$ and hence $a\in {\rm Nil}(R)$. So, $a$ is a nilpotent element, and thus it is strongly nil-clean, as stated.
\end{proof}

\begin{lemma}\label{lem3.15}
Let $R$ be a ring. Then, the following points are equivalent for a semi-local ring $R$:
\\
(i) $R$ is a GSNC ring.
\\
(ii) $R$ is either a local ring with nil $J(R)$, or a strongly nil-clean ring, or $R/J(R)={\rm M}_2(\mathbb{Z}_2)$  with $J(R)$ nil.
\end{lemma}

\begin{proof}
(i) $\Rightarrow$ (ii). Applying Proposition \ref{pro3.7}(ii), we have that $R/J(R)$ is a GSNC ring. Since $R$ is a semi-local ring, we write $R/J(R) = \prod_{i=1}^{m} {\rm M}_{n_i}(D_i)$, where each component ${\rm M}_{n_i}(D_i)$ is a matrix ring over a division ring $D_i$. If, for a moment, $m=1$, then combining Example \ref{exm3.10} and Theorem \ref{thm3.9}, we infer that either $R/J(R) = D_1$ or $R/J(R) = {\rm M}_2(\mathbb{Z}_2)$.

If, however, $m>1$, then employing Proposition \ref{pro3.7}(iii), each of the rings ${\rm M}_n(D)$ must be strongly nil-clean, which means that, for any $i$, the equality ${\rm M}_{n_i}(D_i) = \mathbb{Z}_2$ holds.
\\
(ii) $\Rightarrow$ (i). If $R$ is a strongly nil-clean ring, it is apparent that $R$ is GSNC. If now $R$ is local with nil $J(R)$, then $R$ has to be GSNC invoking Example
\ref{exa3.14}.
\end{proof}

\begin{corollary}\label{cor3.16}
Let $R$ be a ring. Then, the following items are equivalent for a semi-simple ring $R$:
\\
(i) $R$ is a GSNC ring.\\
(ii) $R$ is either a division ring, or a strongly nil-clean ring, or $R={\rm M}_2(\mathbb{Z}_2)$.
\end{corollary}

\begin{proof}
It is immediate by using Lemma \ref{lem3.15}.
\end{proof}	

We now proceed by proving some structural results.

\begin{proposition}\label{pronew}
Let $R$ be a ring. Then, the following two issues are equivalent for an abelian ring $R$:
\\
(i) $R$ is a GSNC ring.
\\
(ii) $R$ is either a local with nil $J(R)$, or a strongly nil-clean ring.
\end{proposition}

\begin{proof}
It is similar to that of Lemma \ref{lem2.10}.
\end{proof}

The next claim also appeared in \cite{5}, but we will give it a bit more conceptual proof.

\begin{proposition}\label{pro3.17}
A ring $R$ is strongly nil-clean if, and only if,
\\
(i) $R$ is UU;
\\
(ii) $R$ is strongly clean.
\end{proposition}

\begin{proof}
"$\Rightarrow$". Let $u\in U(R)$. So, $a=e+q$, where $e\in {\rm Id}(R)$ and $q\in {\rm Nil}(R)$ with $eq=qe$. Thus,
$e=u-q$, where $uq=qu$, and hence $e\in U(R)$. Therefore, $e=1$ whence $u\in 1+{\rm Nil}(R)$. That is why, $R$ is a UU ring.
\\
"$\Leftarrow$". Let $a\in R$. Write $a+1 = e+u$, where $eu=ue$, $e\in {\rm Id}(R)$ and $u\in U(R)$. So, $a=e+(-1+u)$,
where $-1+u\in {\rm Nil}(R)$ and $e(-1+u)=(-1+u)e$, guaranteeing that $a$ is strongly nil-clean, as required.
\end{proof}

\begin{lemma}\label{lem3.18} \cite{21}
A unit $u$ of a ring $R$ is strongly nil-clean if, and only if, $u\in 1+{\rm Nil}(R)$. In particular, $R$ is a UU ring if, and only if, every unit of $R$ is strongly nil-clean.
\end{lemma}

In our terminology alluded to above, we extract the following two assertions:

\begin{corollary}\label{cor3.19}
A ring $R$ is strongly nil-clean if, and only if,
\\
(i) $R$ is UU;
\\
(ii) $R$ is GSNC.
\end{corollary}

\begin{proof}
It follows directly from a combination of Proposition \ref{pro3.17} and Lemma \ref{lem3.18}.
\end{proof}

\begin{corollary}\label{cor3.20}
Let $R$ be a UU ring. Then, the following are equivalent:
\\
(i) $R$ is a strongly clean ring.
\\
(ii) $R$ is a strongly nil-clean ring.
\\
(iii) $R$ is a GSNC ring.
\\
(iv) $R$ is a strongly $\pi$-regular ring.
\end{corollary}

\begin{proposition}\label{pro3.21}
A ring $R$ is GUNC if, and only if,
\\
(i) $R$ is abelian;
\\
(ii) $R$ is GSNC.
\end{proposition}

\begin{proof}
It follows immediately combining Lemma \ref{lem2.10}, Proposition \ref{pronew} and \cite[Theorem $ 5. 3. 3 $]{3}.
\end{proof}

We call a ring an {\it NR} ring if its set of nilpotent elements forms a subring. Recall also that a ring $R$ is called an {\it exchange} ring, provided that, for any $a$ in $R$, there is an idempotent $e\in R$ such that $e\in aR$ and $1-e\in (1-a)R$.

\begin{lemma}\label{lem3.22}
Let $R$ be a NR ring. Then, the following two conditions are equivalent:
\\
(i) $R$ is either local with nil $J(R)$, or $R$ is strongly nil-clean.
\\
(ii) $R$ is a GSNC ring.
\end{lemma}

\begin{proof}
(i) $\Rightarrow$ (ii). It is clear by following Example \ref{exa3.14}.
\\
(ii) $\Rightarrow$ (i). If $R$ is a GSNC ring, then Lemma \ref{lem3.2} applies to get that $R$ is an exchange ring. Therefore, in virtue of \cite[Corollary 2]{26}, $R/J(R)$ is an abelian ring. Consequently, Proposition \ref{pro3.21} enables us that $R/J(R)$ is GUNC ring, which means that either $R/J(R)$ is local, or $R/J(R)$ is an UNC ring. Hence, the application of \cite[Theorem 19]{25} leads us to $R/J(R)$ is either a local ring, or a Boolean ring. Finally, inspired by \cite[Theorem 5.1.5]{3}, we conclude that $R$ is either a local ring, or is a strongly nil-clean ring.
\end{proof}

The next constructions are worthy of mentioning.

\begin{example}\label{exa3.23}
(i) Any strongly nil-clean ring is GSNC, but the converse is {\it not} true in general. For instance, consider the ring $\mathbb{Z}_{3}$ which is GSNC, but is {\it not} strongly nil-clean.
\\
(ii) Any GSNC ring is strongly clean, but the converse is {\it not} generally valid. For instance, the ring
$\mathbb{Z}_{2}[[x]]$ is strongly clean, but is {\it not} GSNC.
\\
(iii) Any GUNC ring is GSNC, but the converse is {\it not} fulfilled in generality. For instance, the ring ${\rm M}_{2}(\mathbb{Z}_{2})$ is GSNC, but is {\it not} GUNC.
\end{example}

A ring $R$ is said to be {\it semi-potent} if every one-sided ideal not contained in $J(R)$ contains a non-zero idempotent. Additionally, a semi-potent ring $R$ is called {\it potent}, provided all of its idempotents lift modulo
$J(R)$. Notice that semi-potent rings and potent rings were also named in \cite{10} as {\it $I_{0}$-rings} and {\it $I$-rings}, respectively.

In the terminology we have introduced, we remember that the definitions of GUSC, GUC, GUNC and GSNC rings are given above.

\begin{proposition}\label{pro3.24}
Let $R$ be a ring, ${\rm Id}(R)=\{0,1\}$ and $J(R)$ is nil. Then, the following are equivalent:
\\
(i) $R$ is a local ring.
\\
(ii) $R$ is a GUSC ring.
\\
(iii) $R$ is a strongly clean ring.
\\
(iv) $R$ is a clean ring.
\\
(v) $R$ is an exchange ring.
\\
(vi) $R$ is a potent ring.
\\
(vii) $R$ is a semi-potent ring.
\\
(viii) $R$ is a GUC ring.
\\
(ix) $R$ is a GUNC ring.
\\
(x) $R$ is a GSNC ring.
\end{proposition}

\begin{proof}
(i) $\Rightarrow$ (ii). It follows from \cite[Example 2.7]{20}.

(ii) $\Rightarrow$ (iii). It follows from \cite[Lemma 3.2]{20}.

(iii) $\Rightarrow$ (iv) $\Rightarrow$ (v) $\Rightarrow$ (vi) $\Rightarrow$ (vii). These implications are pretty obvious, so leave all details to the interested reader.

(vii) $\Rightarrow$ (i). It is obvious since  ${\rm Id}(R) = \{0,1\}$.

(i) $\Leftrightarrow$ (viii). It follows at once from \cite[Proposition 2.9]{2}.

(viii) $\Leftrightarrow$ (ix). It follows directly from Example \ref{cor2.11}.

(ix) $\Leftrightarrow$ (x). It follows immediately from Proposition \ref{pro3.21}.
\end{proof}

\begin{example}\label{exa3.25}
(i) For any ring $R$, the power series ring $R[[x]]$ is {\it not} GSNC.
\\
(ii) If $R$ is a ring, then the polynomial ring $R[x]$ is {\it not} GSNC.
\end{example}

\begin{proof}
(i) Note the principal fact that the Jacobson radical of $R[[x]]$ is {\it not} nil. Thus, in view of Lemma
\ref{lem3.4}, $R[[x]]$ need {\it not} be a GSNC ring.
\\
(ii) If we assume the contrary that $R[x]$ is GSNC, then $R[x]$ is clean in accordance with Lemma \ref{lem3.2}.
This, however, contradicts \cite[Example $ 2 $]{4}.
\end{proof}

\begin{example}\label{exa3.26}
Let $R$ be a ring, and let
\begin{center}
${\rm S}_{n}(R)=\left\lbrace (a_{ij})\in {\rm T}_{n}(R)\, | \, a_{11}=a_{22}=\cdots=a_{nn}\right\rbrace.  $
\end{center}
Then, $R$ is GSNC if, and only if, ${\rm S}_{n}(R)$ is GSNC for all $n\in \mathbb{N}$.
\end{example}

\begin{proof}
Given ${\rm S}_{n}(R)$ is GSNC. It is easy to see that $R$ is a GSNC ring, so we omit the details.

Conversely, let $R$ be GSNC. Thus, for any $(a_{ij}) \in {\rm S}_{n}(R)$, where $(a_{ij})\notin U({\rm S}_{n}(R))$,
we see that $a_{11}=\cdots =a_{nn}\notin U(R)$. So, Proposition \ref{pro3.5} allows us to infer that
$$a_{ii}-a_{ii}^{2}\in {\rm Nil}(R)$$ for each $i$. Furthermore, write $(a_{ii}-a_{ii}^{2})^m=0$ for some
$m\in \mathbb{N}$. Consequently, $\big((a_{ij})-(a_{ij})^{2}\big)^{nm}=0$. Now, according to Proposition
\ref{pro3.5}, we infer that ${\rm S}_{n}(R)$ is GSNC, as stated.
\end{proof}

Let $R$ be a ring, and define the following rings thus:
\begin{align*}
&{\rm A}_{n,m}(R) =R[x,y | x^n=xy=y^m=0], \\
&{\rm B}_{n,m}(R) =R\left\langle x,y | x^n=xy=y^m=0  \right\rangle, \\
&{\rm C}_{n}(R) =R \langle x,y | x^2=\underbrace{xyxyx...}_{\text{$n-1$ words}}=y^2=0 \rangle.
\end{align*}

On the other vein, Wang introduced in \cite{13} the matrix ring ${\rm S}_{n,m}(R)$ as follows: suppose $R$ is a ring, then the matrix ring ${\rm S}_{n,m}(R)$ is representable like this

$$\left\{ \begin{pmatrix}
   a & b_1 & \cdots & b_{n-1} & c_{1n} & \cdots & c_{1 n+m-1}\\
   \vdots  & \ddots & \ddots & \vdots & \vdots & \ddots & \vdots \\
   0 & \cdots & a & b_1 & c_{n-1,n} & \cdots & c_{n-1,n+m-1} \\
   0 & \cdots & 0 & a & d_1 & \cdots & d_{m-1} \\
   \vdots  & \ddots & \ddots & \vdots & \vdots & \ddots & \vdots \\
   0 & \cdots & 0 & 0  & \cdots & a & d_1 \\
   0 & \cdots & 0 & 0  & \cdots & 0 & a
\end{pmatrix}\in {\rm T}_{n+m-1}(R) : a, b_i, d_j,c_{i,j} \in R \right\}.$$

Also, let ${\rm T}_{n,m}(R)$ be

$$\left\{ \left(\begin{array}{@{}c|c@{}}
  \begin{matrix}
  a & b_1 & b_2 & \cdots & b_{n-1} \\
  0 & a & b_1 & \cdots & b_{n-2} \\
  0 & 0 & a & \cdots & b_{n-3} \\
  \vdots & \vdots & \vdots & \ddots & \vdots \\
  0 & 0 & 0 & \cdots & a
  \end{matrix}
  & \bigzero \\
\hline
  \bigzero &
  \begin{matrix}
  a & c_1 & c_2 & \cdots & c_{m-1} \\
  0 & a & c_1 & \cdots & c_{m-2} \\
  0 & 0 & a & \cdots & c_{m-3} \\
  \vdots & \vdots & \vdots & \ddots & \vdots \\
  0 & 0 & 0 & \cdots & a
  \end{matrix}
\end{array}\right)\in {\rm T}_{n+m}(R) : a, b_i,c_j \in R \right\}, $$

as well as we state

$${\rm U}_{n}(R)=\left\{ \begin{pmatrix}
   a & b_1 & b_2 & b_3 & b_4 & \cdots & b_{n-1} \\
   0 & a & c_1 & c_2 & c_3 & \cdots & c_{n-2} \\
   0 & 0 & a & b_1 & b_2 & \cdots & b_{n-3} \\
   0 & 0 & 0 & a & c_1 & \cdots & c_{n-4} \\
   \vdots & \vdots & \vdots & \vdots &  &  & \vdots \\
   0 &0 & 0 & 0 & 0 & \cdots & a
\end{pmatrix}\in {\rm T}_{n}(R) :  a, b_i, c_j \in R \right\}.$$

We now show in the following statement some of the existing relations between these rings.

\begin{lemma} \label{lem3.27}
Let $R$ be a ring and $m, n \in \mathbb{N}$. Then, the following three isomorphisms hold:

(i) ${\rm A}_{n,m}(R) \cong {\rm T}_{n,m}(R)$.

(ii) ${\rm B}_{n,m}(R) \cong {\rm S}_{n,m}(R)$.

(iii) ${\rm C}_{n}(R) \cong {\rm U}_{n}(R)$.
\end{lemma}

\begin{proof}
(i) We assume $f = a+ \sum_{i=1}^{n-1}b_ix^i + \sum_{j=1}^{m-1}c_jy^j \in {\rm A}_{n,m}(R)$. We define $\varphi:  {\rm A}_{n,m}(R) \to {\rm T}_{n,m}(R)$ as
$$\varphi(f)=
\left(\begin{array}{@{}c|c@{}}
  \begin{matrix}
  a & b_1 & b_2 & \cdots & b_{n-1} \\
  0 & a & b_1 & \cdots & b_{n-2} \\
  0 & 0 & a & \cdots & b_{n-3} \\
  \vdots & \vdots & \vdots & \ddots & \vdots \\
  0 & 0 & 0 & \cdots & a
  \end{matrix}
  & \bigzero \\
\hline
  \bigzero &
  \begin{matrix}
  a & c_1 & c_2 & \cdots & c_{m-1} \\
  0 & a & c_1 & \cdots & c_{m-2} \\
  0 & 0 & a & \cdots & c_{m-3} \\
  \vdots & \vdots & \vdots & \ddots & \vdots \\
  0 & 0 & 0 & \cdots & a
  \end{matrix}
\end{array}\right). $$
It can easily be shown that $\varphi$ is a ring isomorphism, as required.

\medskip

(ii) We assume $f \in {\rm B}_{n,m}(R)$ such that
\begin{align*}
    f &=  a_{00}y^0x^0 + a_{01}y^0x^1 + \cdots + a_{0,n-1}y^0x^{n-1}  \\
      &+  a_{10}y^1x^0 + a_{11}y^1x^1 + \cdots + a_{1,n-1}y^1x^{n-1}   \\
      & \qquad \vdots   \qquad \qquad \vdots  \qquad \qquad \qquad \qquad  \vdots   \\
      &+ a_{m-1,0}y^{m-1}x^0 + a_{m-1,1}y^{m-1}x^1 + \cdots + a_{m-1,n-1}y^{m-1}x^{n-1}    \\
\end{align*}
We define $\psi : {\rm B}_{n,m}(R) \to {\rm S}_{n,m}(R)$ as
$$\psi(f)=
 \begin{pmatrix}
   a_{00} & a_{10} & \cdots & a_{m-1,0} & a_{m-1,1} & \cdots & a_{m-1,n-1}\\
   \vdots  & \ddots & \ddots & \vdots & \vdots & \ddots & \vdots \\
   0 & \cdots & a_{00} & a_{10} & a_{11} & \cdots & a_{1,n-1} \\
   0 & \cdots & 0 & a_{00} & a_{01} & \cdots & a_{0,n-1} \\
   \vdots  & \ddots & \ddots & \vdots & \vdots & \ddots & \vdots \\
   0 & \cdots & 0 & 0  & \cdots & a_{00} & a_{0,1} \\
   0 & \cdots & 0 & 0  & \cdots & 0 & a_{00}
\end{pmatrix}.$$
It can plainly be proved that $\psi$ is an isomorphism, as needed.

\medskip

(iii) We introduce the coefficients as follows:
$$f= \sum_{0\le i_j \le 1\atop 1 \le j \le n-1}d_{(i_1, \dots , i_{n-1})}\underbrace{y^{i_1}x^{i_2}y^{i_3}x^{i_4}...}_{\text{$n-1$ words}}\in {\rm C}_{n}(R)$$

We define $\phi : {\rm C}_{n}(R) \to {\rm S}_{n,m}(R)$ as
$$\phi(f)=
\begin{pmatrix}
   d_{(0,0,0, \dots,0)} & d_{(1,0,0, \dots,0)} & d_{(1,1,0, \dots,0)} & d_{(1,1,1, \dots,0)} & \cdots & d_{(1,1,1, \dots,1)}\\
   0 & d_{(0,0,0, \dots,0)} & d_{(0,1,0, \dots,0)} & d_{(0,1,1, \dots,0)} & \cdots & d_{(0,1,1, \dots,1)} \\
   0 & 0 & d_{(0,0,0, \dots,0)} & d_{(1,0,0, \dots,0)} & \cdots & d_{(1,\dots,1,0,0)}\\
   0 & 0 & 0 & d_{(0,0,0, \dots,0)} & \cdots & d_{(0,1,\dots,1,0,0)} \\
   \vdots  & \vdots & \vdots & \vdots & \ddots & \vdots \\
   0 &  0 & 0  & 0  & \cdots & d_{(0,0,0, \dots,0)}
\end{pmatrix}.$$
It can readily be checked that $\phi$ is an isomorphism, as asked for.
\end{proof}

Our three concrete examples are the following.

\begin{example}\label{exa3.28}
Letting $R$ be a ring, then we have:

(i) $R\left[ x,y | x^2=xy=y^2=0  \right] \cong \left\{
\begin{pmatrix}
    a_1 & a_2 & 0 & 0 \\
    0 & a_1 & 0 & 0 \\
    0 & 0 & a_1 & a_3 \\
    0 & 0 & 0 & a_1
\end{pmatrix} : a_i \in R \right\}.
$

(ii) $R\left\langle x,y | x^2=xy=y^2=0  \right\rangle \cong \left\{
\begin{pmatrix}
    a_1 & a_2 & a_3 \\
    0 & a_1 & a_4 \\
    0 & 0 & a_1
\end{pmatrix} : a_i \in R \right\}.$

(iii) $R\left\langle x,y | x^2=xyx=y^2=0  \right\rangle \cong \left\{
\begin{pmatrix}
    a_1 & a_2 & a_3 & a_4 \\
    0 & a_1 & a_5 & a_6 \\
    0 & 0 & a_1 & a_2 \\
    0 & 0 & 0 & a_1
\end{pmatrix} : a_i \in R \right\} \cong {\rm T}({\rm T}(R,R), {\rm M}_2(R)).$
\end{example}

\begin{example}\label{exa3.29}
Let $R$ be a ring. Then, the following statements are equivalent:
\\
(i) $R$ is a GSNC ring.
\\
(ii) ${\rm S}_{n,m}(R)$ is a GSNC ring.
\\
(iii) ${\rm T}_{n,m}(R)$ is a GSNC ring.
\\
(iv) ${\rm U}_{n}(R)$ is a GSNC ring.
\end{example}

\begin{proof}
The proof is similar to that of Example \ref{exa3.26}, so remove the details. 	
\end{proof}

\begin{example}\label{exa3.30}
Let $R$ be a ring. If ${\rm T}_{n}(R)$ is a GSNC ring, then $R$ is GSNC. However, the converse is {\it not} true in general.
\end{example}

\begin{proof}
Choose $e={\rm diag}(1,0,\ldots,0)\in {\rm T}_{n}(R)$. Then, one sees that $R\cong e{\rm T}_{n}(R)e$. Furthermore, with Corollary \ref{cor3.8} at hand, we are done.

As for the converse, take $R=\mathbb{Z}_{3}$ and $S={\rm T}_{2}(\mathbb{Z}_{3})$.
Clearly, $R$ is a GSNC ring. But, an easy inspection leads to $A\hspace{-1mm}=$$0~0\choose 0~2$$\notin U(S)$, and thereby $A$ is definitely {\it not} a strongly nil-clean element in $S$, as required, whence $S$ need {\it not} be GSNC.
\end{proof}

Before stating and proving our next major result, the following two propositions are pretty welcome.

\begin{proposition}\label{pro3.31}
Let $R$ be a GSNC ring. Then, for any $n>2$, there does not exist $0\neq e\in {\rm Id}(R)$ such that $eRe\cong {\rm M}_{n}(S)$ for some ring $S$.
\end{proposition}

\begin{proof} Let us assume the opposite, namely that there exists $0\neq e\in {\rm Id}(R)$ such that $eRe\cong {\rm M}_{n}(S)$ for some ring $S$. Since $R$ is, by assumption, GSNC, it follows from Corollary \ref{cor3.8} that the corner subring $eRe$ is GSNC too, and hence ${\rm M}_{n}(S)$ is GSNC as well. This, however, is a contradiction with Theorem \ref{thm3.9}.
\end{proof}

Recall that a set $\{e_{ij} : 1 \le i, j \le n\}$ of non-zero elements of $R$ is said to be a {\it system of $n^2$ matrix units} if $e_{ij}e_{st} = \delta_{js}e_{it}$, where $\delta_{jj} = 1$ and $\delta_{js} = 0$ for $j \neq s$. In this case, $e := \sum_{i=1}^{n} e_{ii}$ is an idempotent of $R$ and $eRe \cong {\rm M}_n(S)$, where $$S = \{r \in eRe : re_{ij} = e_{ij}r, \text{for all} i, j = 1, 2, . . . , n\}.$$

\begin{proposition}\label{pro3.32}
Every GSNC ring is directly finite.
\end{proposition}

\begin{proof}
Suppose $R$ is a GSNC ring. If we assume the reverse, namely that $R$ is {\it not} a directly finite ring, then there exist elements $a, b \in R$ such that $ab = 1$ but $ba \neq 1$. Putting $e_{ij} := a^i(1-ba)b^j$ and
$e :=\sum_{i=1}^{n}e_{ii}$, a routine verification shows that there will exist a non-zero ring $S$ such that $eRe \cong {\rm M}_n(S)$. However, according to Corollary \ref{cor3.8}, the corner $eRe$ is a GSNC ring, so that ${\rm M}_n(S)$ must also be a GSNC ring, thus contradicting Theorem \ref{thm3.9}.
\end{proof}	

We now have all the machinery necessary to establish the following.

\begin{theorem}\label{thm3.33}
Let $R$ be a ring. Then, the following equivalencies hold:
\\
(i) $R$ is GSNC.
\\
(ii) $\dfrac{R[[x]]}{\langle x^{n}\rangle}$ is GSNC for all $n\in \mathbb{N}$.
\\
(iii) $\dfrac{R[[x]]}{\langle x^{n}\rangle}$ is GSNC for some $n\in \mathbb{N}$.
\end{theorem}

\begin{proof}
(i) $\Rightarrow$ (ii). Set $S:=\dfrac{R[[x]]}{\langle x^{n}\rangle}$.
Thus, $$S=\{\sum_{i=0}^{n-1} r_{i}x^{i}\, | \, r_{0},\ldots,r_{n-1}\in R,\, x^{n}=0 \}.$$
Let $r(x)=\sum_{i=0}^{n-1} r_{i}x^{i}\in S$, where $r(x)\notin U(S)$.
On the other hand, we know that $$U(S)=\{ \sum_{i=0}^{n-1} r_{i}x^{i}\in S \, |\, r_{0}\in U(R) \}.$$
We also see that $$r(x)-r^{2}(x)=(r_{0}-r_{0}^{2})+bx$$ for some $b\in R$. As $r_{0}-r_{0}^{2}\in {\rm Nil}(R)$,
we can find some $m\in \mathbb{N}$ such that $(r_{0}-r_{0}^{2})^{m}=0$, and so
$$(r(x)-r^{2}(x))^{2m+1}=((r_{0}-r_{0}^{2})+bx)^{2m+1}=0.$$
Therefore, one infers that $r(x)-r^{2}(x)\in {\rm Nil}(S)$. Furthermore, we apply Proposition \ref{pro3.5} to get the assertion.
\\
(ii) $\Rightarrow$ (iii). It is trivial.
\\
(iii) $\Rightarrow$ (i). For any $r\in R$, where $r\notin U(R)$, we see that $r\in S$ with $r\notin U(S)$, and thus
$r-r^{2}\in {\rm Nil}(S)$. This implies that $r-r^{2}\in {\rm Nil}(R)$. Therefore, $R$ is GSNC, as asserted.
\end{proof}

An immediate consequence is the one:

\begin{corollary}\label{cor3.34}
Let $R$ be a ring. Then, the following are equivalent:
\\
(i) $R$ is GSNC.
\\
(ii) $\dfrac{R[x]}{\langle x^{n}\rangle}$ is GSNC for all $n\in \mathbb{N}$.
\\
(iii) $\dfrac{R[x]}{\langle x^{n}\rangle}$ is GSNC for some $n\in \mathbb{N}$.
\end{corollary}

Furthermore, let $A, B$ be two rings, and let $M,N$ be the $(A,B)$-bi-module and $(B,A)$-bi-module, respectively. Also, we consider the bilinear maps $\phi : M\otimes_B N \to A$ and $\psi : N\otimes_AM \to B$ that apply to the following properties
$${\rm Id}_M \otimes_B \psi = \phi \otimes_A {\rm Id}_M, \quad {\rm Id}_N \otimes_A \phi = \psi \otimes_B {\rm Id}_N.$$
For $m \in M$ and $n \in N$, we define $mn := \phi(m \otimes n)$ and $nm := \psi(n \otimes m)$.
Thus, the $4$-tuple
$R= \begin{pmatrix}
	A & M \\
	N & B
\end{pmatrix}$
becomes to an associative ring equipped with the obvious matrix operations, which is called a {\it Morita context ring}. Denote the two-sided ideals ${\rm Im}\phi$ and ${\rm Im}\psi$ to $MN$ and $NM$, respectively, that are called the {\it trace ideals} of the Morita context.

We now have all the ingredients needed to prove the following.

\begin{theorem}\label{thm3.35}
Let $R=$$A~M \choose N~B$ be a Morita context such that $MN$ and $NM$ are nilpotent ideals of $A$ and $B$, respectively. Then, $R$ is GSNC if, and only if, both A and B are strongly nil-clean.
\end{theorem}

\begin{proof}
Assume $R$ is a GSNC ring. We show that $A$ is a strongly nil-clean ring and, similarly, it can be shown that $B$ is also a strongly nil-clean ring. To this goal, let us assume $a \in A$; then an elementary check gives that
$C=\begin{pmatrix}
    a & 0 \\ 0 & 0
\end{pmatrix} \notin U(R)$, so Proposition \ref{pro3.5} yields $C-C^2 \in {\rm Nil}(R)$, that is, $a-a^2 \in {\rm Nil}(A)$. So, $A$ is a strongly nil-clean ring. The converse implication can be obtained by \cite[Theorem 3.4]{23}.
\end{proof}

Now, let $R$, $S$ be two rings, and let $M$ be an $(R,S)$-bi-module such that the operation $(rm)s = r(ms$) is valid for all $r \in R$, $m \in M$ and $s \in S$. Given such a bi-module $M$, we can set

$$
{\rm T}(R, S, M) =
\begin{pmatrix}
	R& M \\
	0& S
\end{pmatrix}
=
\left\{
\begin{pmatrix}
	r& m \\
	0& s
\end{pmatrix}
: r \in R, m \in M, s \in S
\right\},
$$
where it forms a ring with the usual matrix operations. The so-stated formal matrix ${\rm T}(R, S, M)$ is called a {\it formal triangular matrix ring}. In Theorem \ref{thm3.35}, if we set $N =\{0\}$, then we will obtain the following two corollaries.

\begin{corollary}\label{cor3.36}
Let $R,S$ be rings and let $M$ be an $(R,S)$-bi-module. Then, the formal triangular matrix ring ${\rm T}(R,S,M)$
is GSNC if, and only if, both $A$ and $B$ are strongly nil-clean.
\end{corollary}

\begin{corollary}\label{cor3.37}
Let $R$ be a ring and $n\geqslant 1$ is a natural number. Then, ${\rm T}_{n}(R)$ is GSNC if, and only if, $R$ is strongly nil-clean.
\end{corollary}

Given now a ring $R$ and a central element $s$ of $R$, the $4$-tuple
$\begin{pmatrix}
	R& R \\
	R& R
\end{pmatrix}$
becomes a ring with addition defined componentwise and with multiplication defined by
$$
\begin{pmatrix}
	a_1& x_1 \\
	y_1& b_1
\end{pmatrix}
\begin{pmatrix}
	a_2& x_2 \\
	y_2& b_2
\end{pmatrix}=
\begin{pmatrix}
	a_1a_2 + sx_1y_2& a_1x_2 + x_1b_2 \\
	y_1a_2 + b_1y_2& sy_1x_2 + b_1b_2
\end{pmatrix}.
$$
This ring is denoted by ${\rm K}_s(R)$. A Morita context
$
\begin{pmatrix}
	A& M \\
	N& B
\end{pmatrix}
$
with $A = B = M = N = R$ is called a {\it generalized matrix ring} over $R$. It was observed in \cite{11} that a ring $S$ is a generalized matrix ring over $R$ if, and only if, $S = {\rm K}_s(R)$ for some $s \in {\rm Z}(R)$, the center of $R$. Here $MN = NM = sR$, so that $$MN \subseteq J(A) \Longleftrightarrow  s \in J(R), NM \subseteq J(B) \Longleftrightarrow  s \in  J(R),$$ and $MN, NM$ are nilpotent $\Longleftrightarrow  s$ is a nilpotent. Thus, Theorem \ref{thm3.35} has the following consequence, too.

\begin{corollary}\label{cor3.38}
Let $R$ be a ring and $s\in {\rm Z}(R) \cap {\rm Nil}(R)$. Then, ${\rm K}_{s}(R)$ is GSNC if, and only if, $R$ is strongly nil-clean.
\end{corollary}

Following Tang and Zhou (cf. \cite{12}), for $n\geq 2$ and for $s\in {\rm Z}(R)$, the $n\times n$ formal matrix ring over $R$ defined with the help of $s$, and denoted by ${\rm M}_{n}(R;s)$, is the set of all $n\times n$ matrices over $R$ with usual addition of matrices and with multiplication defined below:

\noindent For $(a_{ij})$ and $(b_{ij})$ in ${\rm M}_{n}(R;s)$, set
$$(a_{ij})(b_{ij})=(c_{ij}), \quad \text{where} ~~ (c_{ij})=\sum s^{\delta_{ikj}}a_{ik}b_{kj}.$$
Here, $\delta_{ijk}=1+\delta_{ik}-\delta_{ij}-\delta_{jk}$, where $\delta_{jk}$, $\delta_{ij}$, $\delta_{ik}$ are the standard {\it Kroncker delta} symbols.

\medskip

Thereby, we arrive at the following.

\begin{corollary}\label{cor3.39}
Let $R$ be a ring and $s\in {\rm Z}(R)$. If ${\rm M}_{n}(R;s)$ is GSNC, then $R$ is GSNC and $s\in J(R)$.
The converse holds provided $R$ is strongly nil-clean and $s\in {\rm Nil}(R)$.
\end{corollary}

Let $R$ be a ring and $M$ a bi-module over $R$. The {\it trivial extension} of $R$ and $M$ is defined as
$$T(R, M) = \{(r, m) : r \in R \text{and} m \in M\},$$
with addition defined componentwise and multiplication defined by
$$(r, m)(s, n) = (rs, rn + ms).$$
Notice that the trivial extension ${\rm T}(R, M)$ is isomorphic to the subring
$$\left\{ \begin{pmatrix} r & m \\ 0 & r \end{pmatrix} : r \in R \text{ and } m \in M \right\}$$
consisting of the formal $2 \times 2$ matrix rings $\begin{pmatrix} R & M \\ 0 & R \end{pmatrix}$ and, in particular, the isomorphism ${\rm T}(R, R) \cong R[x]/\left\langle x^2 \right\rangle$ is fulfilled. We also note that the set of units of the trivial extension ${\rm T}(R, M)$ is $$U({\rm T}(R, M)) = {\rm T}(U(R), M).$$

A Morita context is referred to as {\it trivial} if the context products are trivial, meaning that $MN = \{0\}$ and $NM = \{0\}$ (see, for instance, \cite[p. 1993]{14}). In this case, we have the isomorphism
$$\begin{pmatrix} A & M \\ N & B \end{pmatrix} \cong {\rm T}(A \times B, M\oplus N),$$
where $\begin{pmatrix} A & M \\ N & B \end{pmatrix}$ represents a trivial Morita context, as stated in \cite{15}.

We, thus, come to the following symmetric relationship.

\begin{lemma}\label{lem3.40}
Let $R$ be a ring and $M$ a bi-module over $R$. Then,
$${\rm Nil}({\rm T}(R, M)) = {\rm T}({\rm Nil}(R), M).$$
\end{lemma}

\begin{proof}
It is technically straightforward, so we drop off the full details leaving them to the interested reader for an inspection.
\end{proof}	
	
A good information gives also the following necessary and sufficient condition.

\begin{proposition}\label{pro3.41}
Let $R$ be a ring and $M$ a bi-module over $R$. Then, ${\rm T}(R, M)$ is GSNC if, and only if, $R$ is GSNC.
\end{proposition}

\begin{proof}
{\bf Method 1:} Assuming $I = (0, M)$, then clearly $I$ is a nil-ideal of the ring ${\rm T}(R, M)$. Moreover, since the isomorphism $R \cong {\rm T}(R, M)/I$ is true, Proposition \ref{pro3.7}(i) employs to get the claim.

{\bf Method 2:} Let ${\rm T}(R, M)$ be a GSNC ring and $a \notin U(R)$. Then, one verifies that $(a,0) \notin U({\rm T}(R, M))$, so Proposition \ref{pro3.5} applies to detect that $(a,0)-(a,0)^2 \in {\rm Nil}({\rm T}(R, M))$, hence
$a-a^2 \in {\rm Nil}(R)$, as required.

Conversely, assuming $R$ is a GSNC ring and $(a, m) \notin U({\rm T}(R, M))$, we derive $a \notin U(R)$. Consequently, $a-a^2 \in {\rm Nil}(R)$, implying $$(a, m)-(a, m)^2 \in {\rm Nil}({\rm T}(R, M)),$$ as needed.
\end{proof}

The next criterion is also worthy of documentation.

\begin{corollary}\label{cor3.38}
Let $R=$$A~M \choose N~B$ be a trivial Morita context. Then, $R$ is GSNC if, and only if, both $A$ and $B$ are strongly nil-clean.
\end{corollary}

\begin{proof}
It is straightforward bearing in mind Propositions \ref{pro3.41} and \ref{pro3.7}(iii).
\end{proof}

Likewise, we can derive the following:

\begin{corollary}\label{cor3.43}
Let $R$ be a ring and $M$ a bi-module over $R$. Then, the following four statements are equivalent:
\begin{enumerate}
\item
$R$ is a GSNC ring.
\item
${\rm T}(R, M)$ is a GSNC ring.
\item
${\rm T}(R, R)$ is a GSNC ring.
\item
$\dfrac{R[x]}{\left\langle x^2 \right\rangle}$ is a GSNC ring.
\end{enumerate}
\end{corollary}

Now, consider $R$ to be a ring and $M$ to be a bi-module over $R$. Let $${\rm DT}(R,M) := \{ (a, m, b, n) | a, b \in R, m, n \in M \}$$ with addition defined componentwise and multiplication defined by $$(a_1, m_1, b_1, n_1)(a_2, m_2, b_2, n_2) = (a_1a_2, a_1m_2 + m_1a_2, a_1b_2 + b_1a_2, a_1n_2 + m_1b_2 + b_1m_2 +n_1a_2).$$ Then, ${\rm DT}(R,M)$ is a ring which is isomorphic to the ring ${\rm T}({\rm T}(R, M), {\rm T}(R, M))$. Also, we have $${\rm DT}(R, M) =
\left\{\begin{pmatrix}
	a &m &b &n\\
	0 &a &0 &b\\
	0 &0 &a &m\\
	0 &0 &0 &a
\end{pmatrix} |  a,b \in R, m,n \in M\right\}.$$ Besides, we also have the following isomorphism as rings: $\dfrac{R[x, y]}{\langle x^2, y^2\rangle} \rightarrow {\rm DT}(R, R)$ defined by $$a + bx + cy + dxy \mapsto
\begin{pmatrix}
	a &b &c &d\\
	0 &a &0 &c\\
	0 &0 &a &b\\
	0 &0 &0 &a
\end{pmatrix}.$$

We, thereby, detect the following.

\begin{corollary}\label{cor3.44}
Let $R$ be a ring and $M$ a bi-module over $R$. Then, the following statements are equivalent:
\begin{enumerate}
\item
$R$ is a GSNC ring.
\item
${\rm DT}(R, M)$ is a GSNC ring.
\item
${\rm DT}(R, R)$ is a GSNC ring.
\item
$\dfrac{R[x, y]}{\langle x^2, y^2\rangle}$ is a GSNC ring.
\end{enumerate}
\end{corollary}

Let $\alpha$ be an endomorphism of $R$ and $n$ a positive integer. It was introduced by Nasr-Isfahani in \cite{30} the {\it skew triangular matrix ring} like this:

$${\rm T}_{n}(R,\alpha )=\left\{ \left. \begin{pmatrix}
	a_{0} & a_{1} & a_{2} & \cdots & a_{n-1} \\
	0 & a_{0} & a_{1} & \cdots & a_{n-2} \\
	0 & 0 & a_{0} & \cdots & a_{n-3} \\
	\ddots & \ddots & \ddots & \vdots & \ddots \\
	0 & 0 & 0 & \cdots & a_{0}
\end{pmatrix} \right| a_{i}\in R \right\}$$
with addition defined point-wise and multiplication given by:
\begin{align*}
	&\begin{pmatrix}
		a_{0} & a_{1} & a_{2} & \cdots & a_{n-1} \\
		0 & a_{0} & a_{1} & \cdots & a_{n-2} \\
		0 & 0 & a_{0} & \cdots & a_{n-3} \\
		\ddots & \ddots & \ddots & \vdots & \ddots \\
		0 & 0 & 0 & \cdots & a_{0}
	\end{pmatrix}\begin{pmatrix}
		b_{0} & b_{1} & b_{2} & \cdots & b_{n-1} \\
		0 & b_{0} & b_{1} & \cdots & b_{n-2} \\
		0 & 0 & b_{0} & \cdots & b_{n-3} \\
		\ddots & \ddots & \ddots & \vdots & \ddots \\
		0 & 0 & 0 & \cdots & b_{0}
	\end{pmatrix}  =\\
	& \begin{pmatrix}
		c_{0} & c_{1} & c_{2} & \cdots & c_{n-1} \\
		0 & c_{0} & c_{1} & \cdots & c_{n-2} \\
		0 & 0 & c_{0} & \cdots & c_{n-3} \\
		\ddots & \ddots & \ddots & \vdots & \ddots \\
		0 & 0 & 0 & \cdots & c_{0}
	\end{pmatrix},
\end{align*}
where $$c_{i}=a_{0}\alpha^{0}(b_{i})+a_{1}\alpha^{1}(b_{i-1})+\cdots +a_{i}\alpha^{i}(b_{i}),~~ 1\leq i\leq n-1
.$$ We denote the elements of ${\rm T}_{n}(R, \alpha)$ by $(a_{0},a_{1},\ldots , a_{n-1})$. If $\alpha $ is the identity endomorphism, then obviously ${\rm T}_{n}(R,\alpha )$ is a subring of upper triangular matrix ring ${\rm T}_{n}(R)$.

We now establish the validity of the following.

\begin{corollary}\label{cor3.45}
Let $R$ be a ring. Then, the following are equivalent:
\begin{enumerate}
\item
$R$ is a GSNC ring.
\item
${\rm T}_{n}(R,\alpha )$ is a GSNC ring.
\end{enumerate}
\end{corollary}

\begin{proof}
Choose
	$$I:=\left\{
	\left.
	\begin{pmatrix}
		0 & a_{12} & \ldots & a_{1n} \\
		0 & 0 & \ldots & a_{2n} \\
		\vdots & \vdots & \ddots & \vdots \\
		0 & 0 & \ldots & 0
	\end{pmatrix} \right| a_{ij}\in R \quad (i\leq j )
	\right\}.$$
Therefore, one easily finds that $I^{n}=(0)$ and $\dfrac{{\rm T}_{n}(R,\alpha )}{I} \cong R$. Consequently, one verifies that Proposition \ref{pro3.7} is applicable to get the pursued result.
\end{proof}

\section{GSNC group rings}

As usual, for an arbitrary ring $R$ and an arbitrary group $G$, the symbol $RG$ stands for the {\it group ring} of $G$ over $R$. Standardly, $\Delta(RG)$ designates the kernel of the classical augmentation map $RG\to R$.

We begin here with the following technicality.

\begin{lemma}\label{lem5.1}
Let $\phi: R \to S$ be a non-zero epimorphism of rings with ${\rm Ker}(\phi) \cap {\rm Id}(R) = \{0\}$. Then, $R$ is a GSNC ring if, and only if, $S$ is a GSNC ring and ${\rm Ker}(\phi)$ is a nil-ideal of $S$.
\end{lemma}

\begin{proof}
Suppose $R$ is a GSNC ring and $a\in {\rm Ker}(\phi)$. Thus, $a \notin U(R)$, so that there exist $e\in {\rm Id}(R)$ and $q\in {\rm Nil}(R)$ with $a = e + q$ and $eq = qe$. That is why, $0 = \phi(a) = \phi(e) + \phi(q)$, yielding
$\phi(e) \in {\rm Id}(S) \cap {\rm Nil}(S) = \{0\}$. This unambiguously shows that $e\in {\rm Id}(R) \cap {\rm Ker}(\phi) = \{0\}$, hence $a = q \in {\rm Nil}(R)$.
	
Next, since $\phi$ is an epimorphism, we have $S \cong R/{\rm Ker}(\phi)$ and, in conjunction with Proposition \ref{pro3.7}(i), we conclude that $R$ is a GSNC ring.

The converse relation can easily be extracted from Proposition \ref{pro3.7}(i).
\end{proof}

The following three lemmas are also useful for further applications.

\begin{lemma}\label{lem5.2}
Let $R$ be a ring and let $G$ be a group, where $\Delta(RG)\cap {\rm Id}(RG)=\{0\}$. Then, $RG$ is a GSNC ring if, and only if, $R$ is a GSNC ring and $\Delta(RG)$ is a nil-ideal of $RG$.
\end{lemma}

\begin{proof}
There is an epimorphism $\varepsilon: RG \to R$ with ${\rm Ker}(\varepsilon)= \Delta(RG)$.
\end{proof}

\begin{lemma}\label{lem5.3} \cite[Lemma $2$]{24}.
Let $p$ be a prime with $p\in J(R)$. If $G$ is a locally finite $p$-group, then $\Delta(RG) \subseteq J(RG)$.
\end{lemma}

\begin{lemma}\label{lem5.4}
Let $R$ be a ring and let $G$ be a locally finite $p$-group, where $p$ is a prime and $p\in J(R)$. Then, $RG$ is a GSNC ring if, and only if, $R$ is a GSNC ring and $\Delta(RG)$ is a nil-ideal of $RG$.
\end{lemma}	

In regard to the last lemma, an important question which could be raised is to find, as in the case of GUNC rings, a suitable criterion when a group ring $RG$ of a locally finite $p$-group $G$ over an arbitrary ring $R$ to be a GSNC ring. In other words, is the restriction $p\in J(R)$ necessary in this claim and whether it could be deduced from the condition $RG$ is GSNC?

\section{Open Questions}

We close the work with the following challenging problems.

\medskip

A ring $R$ is said to be {\it weakly nil-clean} provided that, for any $a \in R$, there exists an idempotent $e \in R$ such that $a-e$ or $a+e$ is nilpotent. Additionally, a ring $R$ is said to be {\it strongly weakly nil-clean} provided $ae=ea$ or, equivalently, provided that, for any $a \in R$, at least one of the elements $a$ or $-a$ is strongly nil-clean (see, e.g., \cite{3, 31}).

We now can formulate the following.

\begin{problem}
Examine those rings whose non-invertible elements are (strongly) weakly nil-clean.
\end{problem}

A ring $R$ is called {\it strongly $2$-nil-clean} if every element in $R$ is the sum of two idempotents and a nilpotent that commute each other (see, for example, \cite{3}). These rings are a common generalization of the aforementioned strongly weakly nil-clean rings.

\medskip

Now, we may raise the following.

\begin{problem}
Examine those rings whose non-invertible elements are strongly $2$-nil-clean.
\end{problem}

\end{document}